\documentclass[11pt,reqno]{amsart}
\pdfoutput=1
\usepackage{amsmath,amsthm,amssymb,comment,fullpage}
\usepackage{braket}
\usepackage{mathtools}
\usepackage{graphicx}
\graphicspath{ {./images/} }

\usepackage{caption}
\usepackage{subcaption}
\usepackage{times}
\usepackage[T1]{fontenc}
\usepackage{mathrsfs}
\usepackage{latexsym}
\usepackage{epsfig}
\usepackage{amsmath,amsfonts,amsthm,amssymb,amscd}
\input amssym.def
\input amssym.tex
\usepackage[dvipsnames]{xcolor}
\usepackage{hyperref}
\usepackage{url}
\newcommand{\bburl}[1]{\textcolor{blue}{\url{#1}}}
\usepackage{mathtools}
\usepackage{tikz}
\usepackage{tkz-tab}
\usepackage{tkz-graph}
\usetikzlibrary{matrix,shapes.geometric,positioning}

\newcommand{\burl}[1]{\textcolor{blue}{\url{#1}}}

\numberwithin{equation}{section}

\newtheorem{thm}{Theorem}[section]

\newtheorem{defi}[thm]{Definition}

\newtheorem{rem}[thm]{Remark}

\theoremstyle{plain}

\newtheorem{definition}[thm]{Definition}

\newtheorem{lemma}[thm]{Lemma}

\theoremstyle{definition}

\newcommand\be{\begin{equation}}
\newcommand\ee{\end{equation}}
\newcommand\bee{\begin{equation*}}
\newcommand\eee{\end{equation*}}
\newcommand\bea{\begin{eqnarray}}
\newcommand\eea{\end{eqnarray}}
\newcommand\beae{\begin{eqnarray*}}
\newcommand\eeae{\end{eqnarray*}}
\newcommand\bi{\begin{itemize}}
\newcommand\ei{\end{itemize}}
\newcommand\ben{\begin{enumerate}}
\newcommand\een{\end{enumerate}}
\newcommand\bc{\begin{center}}
\newcommand\ec{\end{center}}
\newcommand\ba{\begin{array}}
\newcommand\ea{\end{array}}

\newcommand{\B}{\mathcal{B}}

\newcommand{\Z}{\ensuremath{\mathbb{Z}}}

\newcommand{\ep}{\varepsilon}

\newcommand\frakfamily{\usefont{U}{yfrak}{m}{n}}
\DeclareTextFontCommand{\textfrak}{\frakfamily}

\newcommand{\hr}[1]{\href{#1}{\url{#1}}}

\DeclareMathOperator{\Var}{{Var}}
\DeclareMathOperator{\Cov}{{Cov}}
\DeclareMathOperator{\E}{\mathbb{E}}

\title{Hyper-bishops, Hyper-rooks, and Hyper-queens: Percentage of Safe Squares on Higher Dimensional Chess Boards}

\author{Caroline Cashman}
\email{\textcolor{blue}{\href{mailto:cecashman@wm.edu}{cecashman@wm.edu}}}
\address{Department of Mathematics, College of William \& Mary, Willliamsburg, VA}

\author{Joseph Cooper}
\email{\textcolor{blue}{\href{mailto:jc2407@cam.ac.uk}{jc2407@cam.ac.uk}}}
\address{Department of Mathematics, University of Cambridge, UK}

\author{Raul Marquez}
\email{\textcolor{blue}{\href{mailto:raul.marq@yahoo.com}{raul.marq@yahoo.com}}}
\address{Department of Mathematics, University of Texas Rio Grande Valley, Edinburg, TX}

\author[Miller]{Steven J. Miller}
\email{\textcolor{blue}{\href{mailto:sjm1@williams.edu}{sjm1@williams.edu}},  \textcolor{blue}{\href{Steven.Miller.MC.96@aya.yale.edu}{Steven.Miller.MC.96@aya.yale.edu}}}
\address{Department of Mathematics and Statistics, Williams College, Williamstown, MA} 

\author{Jenna Shuffelton}
\email{\textcolor{blue}{\href{mailto:jms13@williams.edu}{jms13@williams.edu}}}
\address{Department of Mathematics and Statistics, Williams College, Williamstown, MA}

\thanks{The authors were supported by Emmanuel College Cambridge, The William \& Mary Charles Center, The Cissy Patterson Fund, Williams College, The Finnerty Fund, and NSF Grant DMS2241623. Special thanks to Warren Johnson for helpful comments on an earlier draft and to the referees for many helpful comments.}

\date{\today}

\begin{document}

\begin{abstract} Chess has inspired an abundance of mathematical problems, especially in combinatorics and probability. One such problem, initially studied by Miller, Sheng and Turek, considers the proportion of safe spaces when randomly placing $n$ rooks on an $n \times n$ chess board. They show that  as $n$ approaches infinity, the proportion of safe spaces converges to $1/e^2$.
We first generalize their results to bishops and queens. This problem is significantly more interesting and difficult; while a rook attacks the same number of spaces regardless of its position, this is not so for bishops and queens. We prove that the proportion of safe spaces on an $n \times n$ board with $n$ randomly placed bishops converges to $2/e^2$, while it converges to $2/e^4$ for $n$ randomly placed queens. 
We then extend this result to the $k$-dimensional chessboard and consider two natural extensions. We define line-pieces to attack along lines, and define hyper-pieces attack along $(k-1)$-dimensional hyper-planes. We provide exact results for line-rooks and hyper-rooks for arbitrary $k$, and give bounds and initial observations for line-bishops, line-queens, hyper-bishops and hyper-queens.
\end{abstract}

\maketitle
\vspace{-.67cm}


\newpage

\section{Introduction}

Chess is a deep well of ideas for mathematical problems, inspiring entire textbooks such as \cite{Wat04} and \cite{Pet11}. A particularly famous problem is the $n$ queens problem, originally proposed by Bezzel under the pen-name ``Schadenfreude" in \cite{Bez48}, which considers the number of ways to place $n$ queens on an $n \times n$ board so that no queen attacks another. This problem has attracted substantial interest, and a survey on related results can be found in \cite{BS09}. While the classical problem counts the number of ways to place non-attacking queens, one natural extension is to consider the maximum number of safe spaces when placing $n$ queens; here the answer is only known for small $n$, as shown in \cite{LV11}. The dominating queens problem asks yet another interesting question, by considering the minimum number of queens required to cover an $n \times n$ chess board.\par
\vspace{.25cm}
Inspired by these problems, Miller, Sheng, and Turek showed in \cite{MST21} that for an $n \times n$ board with $n$ randomly placed rooks, the percentage of safe spaces converges to $1/e^2$ as $n\to \infty$.  Inspired by this paper, we adopt their techniques to several generalizations of the problem, which we outline in Figure \ref{fig:outlineflowchart}. We begin by restating some of their key definitions and notation.
\begin{figure}[h]
    \centering
 \begin{tikzpicture}[scale=.5, every node/.style={rectangle, rounded corners, minimum width=2.5cm, minimum height=2.5cm,text centered, text width=2.5cm, draw=black}]
\node (MST) at (-10,2.5) {\textbf{Miller et. al:} Rooks in 2 Dimensions};
\node (Lem21) at (-3.3,2.5){\textbf{Section 2:} Combinatorial Generalization};
\node (BQ2) at (3.3,2.5){\textbf{Section 3:} Bishops and Queens in 2 Dimensions};
\node (Hyper) at (10,5.75){\textbf{Section 4: } Hyper-pieces in Higher Dimensions};
\node (Line) at (10,-.75){\textbf{Section 5:} Line-pieces in Higher Dimensions};
\draw [->, thick] (MST) -- (Lem21);
\draw [->, thick] (Lem21) -- (BQ2);
\draw [->, thick] (BQ2) -- (Hyper);
\draw [->, thick] (BQ2) -- (Line);
\end{tikzpicture}
\caption{Outline of the Generalizations of \cite{MST21} we present in this paper.}
\label{fig:outlineflowchart}
\end{figure}
\begin{definition} \label{mainnotations}
    A board configuration, denoted $\B$, is a choice of placements of attacking pieces. For a board configuration on a $k$-dimensional chessboard, we define the binary indicator variable $X_{x_1,\ldots,x_k}$ by
    \begin{equation}
        X_{x_1,\ldots,x_k}(\B) \ := \
        \begin{cases}
            1 & (x_1,\ldots,x_k) \;\text{\normalfont is safe under } \B \\
            0 & \text{\normalfont otherwise.}
        \end{cases}
    \end{equation}

 On a $k$-dimensional board with side length $n$, we denote the number of safe spaces on the board configuration $\B$ by $S_{n,k}(\B)$, so
\begin{equation}
    S_{n,k}(\B) \ := \ \sum^n_{x_1, \ldots, x_k = 1} X_{x_1, \ldots, x_k}(\B), \\
\end{equation}
and, using $\E$ to denote the expected value function, 
\begin{equation}
    \E[S_{n,k}] \ = \ \sum^n_{x_1, \ldots, x_k = 1} \E[X_{x_1, \ldots, x_k}(\B)].
\end{equation}\\
\vspace{.1cm}\\
Lastly, we define
\begin{equation}
    \mu_{n,k} \ := \ \frac{1}{n^k}\sum^n_{x_1, \ldots, x_k = 1} \E[X_{x_1, \ldots, x_k}(\B)],
\end{equation}\\
\vspace{.1cm}\\
so $\mu_{n,k}$ is the expected proportion of safe spaces on the board. We omit $k$ in the standard case when $k=2$.
\end{definition}
We first extend the work of Miller, Sheng, and Turek to bishops and queens in $2$ dimensions, and determine the expected percentage of safe spaces. Notably, the problem becomes much more interesting here. While a rook attacks the same number of spaces at any position on the board, bishops and queens are able to attack more spaces when they are placed near the center. As an illustration, note that on an $n \times n$ board, with $n$ odd, a bishop in the center attacks $2(n-1)$ spaces while a bishop placed on the edge attacks $n-1$ spaces. Additionally, the probability of being placed on a space near the center of the board differs from the probability of being placed on an outer space. This leads us to our first two results. \par
\vspace{.25cm}

\begin{thm}\label{bishopsAvg}
    As $n$ approaches infinity, the mean number of safe spaces on an $n \times n$ chessboard with $n$ randomly placed bishops is asymptotically $2n^2/e^2$, and the expected proportion of safe spaces converges to $2/e^2$.
\end{thm}

\vspace{.10cm}

\begin{thm}\label{queensAvg}
    As $n$ approaches infinity, the mean number of safe spaces on an $n \times n$ chessboard with $n$ randomly placed queens is asymptotically $2n^2/e^4$, and the expected proportion of safe spaces converges to $2/e^4$.
\end{thm}

For both of the above problems, we prove that the variance of the random variable $\mu_{n,k}$ tends to 0 as $n$ approaches infinity. Within the context of this problem, this result is expected, given that as $n$ grows large, the probability of any given attacking piece being placed in a more powerful space near the center of the board becomes relatively small.

\begin{thm}\label{piecevariance}
    Let $n,k,m,d,a \in \Z_{>0}$. Define $\mu_{n,k}$ as in Definition \ref{mainnotations}, with $dn^{k-m}$ attacking pieces placed, each of which attack $an^{m}$ spaces.  Then, the variance of the random variable with mean $\mu_{n,k}$ approaches $0$ as $n$ approaches infinity.
\end{thm}

For our results in higher dimensions, we first introduce some notation and terminology.\par
\vspace{.25cm}

\subsection{Notation and Basic Definitions in Higher Dimensions}
As we move into $k$ dimensions, we first provide some definitions. There are a variety of possible extensions, all with their own strengths and weaknesses, so choosing to define pieces or setups differently than we do may prove an interesting path for future work.\par
\vspace{.25cm}

\begin{defi}[Higher dimensional boards]\label{CBdefi}
A $k$-dimensional board has $k$ dimensions with equal integer side length $n$. Boards are created by stacking alternating boards in the $(k-1)$-dimensional subspace so that no two adjacent spaces are the same color.
\end{defi}
Note that this is how a standard $2$-dimensional chessboard is created: spaces of alternating color are placed in a line and the colors are shifted in the line above so that no space is adjacent to a space of the same color. This definition is particularly strong as it guarantees that any subspace of a $k$-dimensional chessboard is still a chessboard. It also ensures that bishops maintain their parity in $k$ dimensions. We illustrate what a $5 \times 5 \times 5$ chessboard would look like in Figure \ref{fig:3dboard}.

\begin{figure}[h!bt]
    \centering
\begin{tikzpicture}[scale=.5]

	\draw[very thick] (0,5,0) -- (0,5,5);
	\draw[very thick] (5,0,0) -- (5,0,5);
	\draw[very thick] (5,5,0) -- (5,5,5);
	\draw[very thick] (0,5,0) -- (5,5,0);
	\draw[very thick] (5,0,0) -- (5,5,0);
     \draw[very thick]	(0,0,5) -- (5,0,5) -- (5,5,5) -- (0,5,5) -- cycle;

\foreach \y in {1,3,5}{	
 \foreach \x in {1,3,5}{
         \draw[fill = black] (\x-1,\y-1,5) -- (\x,\y-1,5) -- (\x,\y,5) -- (\x-1,\y,5) -- cycle;
    }
    }
\foreach \y in {2,4}{	
 \foreach \x in {2,4}{
         \draw[fill = black] (\x-1,\y-1,5) -- (\x,\y-1,5) -- (\x,\y,5) -- (\x-1,\y,5) -- cycle;
    }
    }

\foreach \y in {1,3,5}{	
 \foreach \x in {1,3,5}{
         \draw[fill = black] (\x-1,5,\y-1) -- (\x,5,\y-1) -- (\x,5,\y) -- (\x-1,5,\y) -- cycle;
    }
    }
\foreach \y in {2,4}{	
 \foreach \x in {2,4}{
         \draw[fill = black] (\x-1,5,\y-1) -- (\x,5,\y-1) -- (\x,5,\y) -- (\x-1,5,\y) -- cycle;
    }
    }
		
\foreach \y in {1,3,5}{	
 \foreach \x in {1,3,5}{
         \draw[fill = black] (5,\x-1,\y-1) -- (5,\x,\y-1) -- (5,\x,\y) -- (5,\x-1,\y) -- cycle;
    }
    }
\foreach \y in {2,4}{	
 \foreach \x in {2,4}{
         \draw[fill = black] (5,\x-1,\y-1) -- (5,\x,\y-1) -- (5,\x,\y) -- (5,\x-1,\y) -- cycle;
    }
    }

\end{tikzpicture}
\caption{Depiction of a $5 \times 5 \times 5$ chessboard.}
\label{fig:3dboard}
\end{figure}
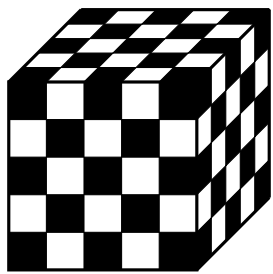
\vspace{.25cm}

We provide two natural extensions of chess pieces in higher dimensions, namely hyper-pieces and line-pieces. Hyper-pieces attack along $(k-1)$-dimensional hyperplanes; for example, a hyper-rook attacks any piece that shares one coordinate with it. Line-pieces attack along lines; for example, a line-rook attacks any piece that shares $(k-1)$ coordinates with it.\par \vspace{.25cm}

We find that for hyper-rooks, after adjusting for their relative board covering power, they cover an equal amount of the board on average. More explicitly, we have the following theorem.
\begin{thm}\label{rooksAvg}
    The mean number of safe spaces on a $k$-dimensional chessboard with side length $n$ and $n$ hyper-rooks placed on it is asymptotically $n^k/e^k$. Hence, the probability a space is safe converges to $1/e^k$.
\end{thm}

We then consider hyper-bishops and hyper-queens in $k$ dimensions. While a rook at any given space attacks the same number of spaces, the number of spaces a bishop attacks varies widely, often with little symmetry. Due to this, we are unable to determine the expected proportion of safe spaces in $k$ dimensions, and instead analyze $3$ dimensions. \par \vspace{.25cm}

Finally, when considering line-pieces, we determine that the expected proportion of safe squares when randomly placing $n^{k-1}$ line-rooks on a $k$-dimensional chessboard is $1/e^k$. We then consider line-bishops and line-queens and provide an integral for estimating the expected proportion of safe squares as $n$ approaches infinity. We computationally determine the expected proportion of safe squares when $k=3$. Our methods here are similar to the 2-dimensional case. \par \vspace{.25cm}

\section{Combinatorial Background}
Many of the results in \cite{MST21} follow from their proof of the fact that \begin{equation*}
    \lim_{n \to \infty} \binom{n^{2}- an -b}{n}\bigg/ \binom{n^{2}}{n} \ = \ \frac{1}{e^{a}}.
\end{equation*}
We generalize their proof to show that a similar statement holds in $k$ dimensions.   \par
\vspace{.25cm}

\begin{lemma}\label{lemma: comboliminf}
    For positive integers $a,k,m,c,d$ and any integer $b$, with $k > m > k-c$, we have
    \begin{equation}
    \lim_{n \to \infty}\binom{n^{k}- an^{m} +bn^{k-c}}{dn^{k-m}} \bigg/ \binom{n^k} {dn^{k-m}} \ = \ \frac{1}{e^{da}}.
    \end{equation}
\end{lemma}\par
\vspace{.25cm} 
\begin{rem}
    This lemma is extremely powerful because of the behavior it describes. Each term can be related to the chess problems.

    \begin{center}
        \begin{tabular}{c|c}
            $n^k$ & \text{Total spaces on chess board.} \\
            \hline
            $an^m -bn^{k-c}$ & \text{Spaces attacked by piece.} \\
            \hline
            $dn^{k-m}$ & \text{Pieces placed on board.} \\
        \end{tabular}
    \end{center}
    We count the number of board setups in which a space is safe and divide by the total number of possible board configurations, ending up with the probability that the space is safe.
\end{rem}
\begin{proof} Note that

\begin{align}
    &\binom{n^{k}- an^{m} +bn^{k-c}}{dn^{k-m}} \bigg/ {\binom{n^k} {dn^{k-m}}}
    \nonumber\\
    &= \ \frac{(n^k- an^{m}+bn^{k-c})!}{(dn^{k-m})!(n^k- an^{m}+bn^{k-c}-dn^{k-m})!} \cdot \frac{(dn^{k-m})!(n^k-dn^{k-m})!}{(n^k)!} \nonumber\\
     &= \ \frac{(n^k- an^{m}+bn^{k-c})!}{(n^k- an^{m}+bn^{k-c}-dn^{k-m})!} \cdot \frac{(n^k-dn^{k-m})!}{(n^k)!}\nonumber\\
    &= \ \frac{(n^k-dn^{k-m})(n^k-dn^{k-m}-1)\cdots(n^k- dn^{k-m}- an^{m}+bn^{k-c}+1)}{(n^k)(n^k-1)\cdots(n^k- an^{m}+bn^{k-c}+1)}\nonumber\\
    &= \ \prod_{i=0}^{an^{m}-bn^{k-c}-1} \frac{n^k-dn^{k-m}-i}{n^k-i}\nonumber \\
    &= \ \prod_{i=0}^{an^{m}- bn^{k-c}-1} \left(1-\frac{dn^{k-m}}{n^k-i}\right).
\end{align}
Knowing that $\lim_{n\to \infty}(1-d/n^m)^{an^m} = 1/e^{da}$, we look to express $(1-(dn^{k-m})/(n^k-i))$ as $(1-d/n^m-\delta)$, where $\delta$ is some small correction. To this end, we note  \be\frac{dn^{k-m}}{n^{k}-i} \ = \ \frac{d}{n^{m}}+ \frac{di}{n^m(n^{k}-i)}. \ee \\
Given that the product is defined for all $i\leq an^{m} -bn^{k-c}-1$, we see that
\begin{multline} \label{hypcombo}
    \left(1-\frac{d}{n^{m}}- \frac{d(an^{m} -bn^{k-c}-1)}{n^m(n^{k}-an^{m}+bn^{k-c}+1)}\right)^{an^{m} - bn^{k-c}}\\
   \leq \ \prod_{i=0}^{an^{m} -bn^{k-c}-1} \left(1-\frac{dn^{k-m}}{n^k-i}\right)
   \ \leq \ \left(1-\frac{d}{n^{m}}\right)^{an^{m} -bn^{k-c}}.
\end{multline}
Thus, \be\label{bpowertozerocombo} \lim_{n \to \infty}\left(1-\frac{d}{n^{m}}\right)^{an^{m}-bn^{k-c}} \ = \ \lim_{n \to \infty}\left(\left(1-\frac{d}{n^{m}}\right)^{n^{m}}\right)^a \cdot \left(\left(1-\frac{d}{n^{m}}\right)^{n^{k-c}}\right)^b \ = \ \frac{1}{e^{da}}. \ee
Therefore, for any $n^{k-c}$ such that $m>k-c$, $\lim_{n \to \infty}((1-\frac{1}{n^{m}})^{n^{k-c}})^b=1$. Hence, any $n$ with degree less than $m$ does not impact the limit.\par
\vspace{.25cm}

\noindent We now consider the lower bound. Factoring out $(1-\frac{d}{n^m})$, whose behavior we understand, we see that what is left approaches 1, and thus does not change the limit. We find
\begin{align}
    &1-\frac{d}{n^{m}}- \frac{d(an^{m}-bn^{k-c}-1)}{n^m(n^{k}-an^{m}-bn^{k-c}-1)}\notag \\ \notag
    &= \ \left(1-\frac{d}{n^{m}}\right)\left(1- \frac{d(an^{m}-bn^{k-c}-1)}{n^m(n^{k}-an^{m}-bn^{k-c}-1)}\cdot \frac{n^{m}}{n^{m}-d}\right) \\
    &= \ \left(1-\frac{d}{n^{m}}\right)\left(1- \frac{d(an^{m}-bn^{k-c}-1)}{n^{k}-an^{m}-bn^{k-c}-1}\cdot \frac{1}{n^{m}-d}\right).
\end{align}
So from \eqref{hypcombo} and substituting with the above equation, we have \begin{align}
    &\left(1-\frac{d}{n^{m}}\right)^{an^{m}-bn^{k-c}}\left(1- \frac{d(an^{m}-bn^{k-c}-1)}{n^{k}-an^{m}-bn^{k-c}-1}\cdot \frac{1}{n^{m}-d}\right)^{an^{m}-bn^{k-c}} \notag \\
   & \leq \ \prod_{i=0}^{an^{m}-bn^{k-c}-1} \left(1-\frac{dn^{k-m}}{n^k-i}\right) \notag \\
    &\leq \ \left(1-\frac{d}{n^{m}}\right)^{an^{m}-bn^{k-c}}.
\end{align}\\
We now show via the Squeeze Theorem that \be \lim_{n \to \infty}\left(1- \frac{d(an^{m}-bn^{k-c}-1)}{n^{k}-an^{m}-bn^{k-c}-1}\cdot \frac{1}{n^{m}-d}\right)^{an^{m}-bn^{k-c}} \ = \ 1. \ee
Trivially, \be \lim_{n \to \infty}\left(1- \frac{d(an^{m}-bn^{k-c}-1)}{n^{k}-an^{m}-bn^{k-c}-1}\cdot \frac{1}{n^{m}-d}\right)^{an^{m}-bn^{k-c}}\leq (1)^{an^{m}-bn^{k-c}} \ = \ 1.\ee
For the other direction, let $\ep >0$. Then, as $n$ approaches infinity, $\ep > \frac{d(an^{m}-bn^{k-c}-1)}{n^{k}-an^{m}-bn^{k-c}-1}$ given that $k>m$. Then, $\frac{\ep}{n^{m}-d} > \frac{d(an^{m}-bn^{k-c}-1)}{(n^{k}-an^{m}-bn^{k-c}-1)(n^{m}-d)}$, so consider
\begin{align}
    \lim_{n \to \infty}\left(1- \frac{d(an^{m}-bn^{k-c}-1)}{n^{k}-an^{m}-bn^{k-c}-1}\cdot \frac{1}{n^{m}-d}\right)&^{an^{m}-bn^{k-c}}
    > \ \lim_{n \to \infty}\left(1- \frac{\ep}{n^{m}-d}\right)^{an^{m}-bn^{k-c}} \notag \\
    =  \ \lim_{n \to \infty}\left(\left(1- \frac{\ep}{n^{m}}\right)^{n^{m}}\right)^{a}\left(\left(1- \frac{\ep}{n^{m}}\right)^{n^{k-c}}\right)^{b}.
\end{align}
Then, \eqref{bpowertozerocombo} gives us
\begin{align}
    \lim_{n \to \infty}\left(\left(1- \frac{\ep}{n^{m}}\right)^{n^{m}}\right)^{a}\left(\left(1- \frac{\ep}{n^{m}}\right)^{n^{k-c}}\right)^{b}  \
    &= \ \lim_{n \to \infty}\left(\left(1- \frac{\ep}{n^{m}}\right)^{n^{m}}\right)^{a}\lim_{n \to \infty}\left(\left(1- \frac{\ep}{n^{m}}\right)^{n^{k-c}}\right)^{b} \notag \\
    &= \ \lim_{n \to \infty}\left(\left(1- \frac{\ep}{n^{m}}\right)^{n^{m}}\right)^{a} \cdot 1 \notag\\
    &= \ e^{-\ep a}.
\end{align}
Then, as $\ep$ approaches $0$, $e^{-\ep a} \to e^0=1$, proving that
\begin{equation}
\lim_{n \to \infty}\left(1- \frac{d(an^{m}-bn^{k-c}-1)}{n^{k}-an^{m}-bn^{k-c}-1}\cdot \frac{1}{n^{m}-d}\right)^{an^{m}-bn^{k-c}} \ = \ 1.
\end{equation}
Therefore, \be\lim_{n \to \infty}\frac{\binom{n^{k}- an^{m} +bn^{k-c}}{dn^{k-m}}}{\binom{n^k} {dn^{k-m}}} \ = \ \frac{1}{e^{da}}.\ee
\end{proof}

\section{Bishops and Queens in 2 Dimensions}
\subsection{Probability of a Safe Square for Bishops and Queens}
While rooks attack the same number of spaces regardless of their board position, the same is not true for bishops or queens. Recall that the expected percentage of safe spaces is
\begin{equation}
    \mu_n \ := \ \frac{1}{n^2}\sum^n_{i,j = 1} \E[X_{i,j}].
\end{equation}
Note that the expected number of safe spaces is different for each value of $(i,j)$.\par
\vspace{.25cm}

We first consider the odd board case, although we show later that parity becomes irrelevant as $n$ approaches infinity. We define the number of spaces a bishop attacks in terms of $r$ rings. We define the $0$\textsuperscript{th} ring to contain the space $((n-1)/2, (n-1)/2)$ and recursively define $(r+1)\textsuperscript{st}$ ring to contain the spaces along the border of the $r$\textsuperscript{th} ring, as demonstrated by the colored rings in Figure \ref{fig:ringimage}. Then the outermost ring is at $r=(n-1)/2$. \\
\begin{figure}[h!tb]
    \centering
\begin{tikzpicture}[scale=.75]
\filldraw [color = blue, fill = blue!10, ultra thick] (7,-7) rectangle (0,0);
\filldraw [color = green, fill = green!10,ultra thick] (6,-6) rectangle (1,-1);
\filldraw [color = orange, fill = orange!10,,ultra thick] (5,-5) rectangle (2,-2);
\filldraw [color = red, fill = red!10,ultra thick] (4,-4) rectangle (3,-3);
\draw (0,0) rectangle (7,-7);
    \foreach \row in {0,1, ..., 6} {
        \foreach \column in {0, ..., 2} {
    \fill[black] ({2*\column + mod(\row,2)}, -\row) rectangle +(1,-1);
        }
        \foreach \column in {6} {
        \foreach \row in {0,2,4,6}
    \fill[black] ({\column}, -\row) rectangle +(1,-1);
        }
    }
\draw[red, ultra thick, ->] (3,-7)--(0,-4);
\draw[red, ultra thick, ->] (2,-7)--(7,-2);
\end{tikzpicture}
\caption{Depiction of the rings on a 7 by 7 chessboard, as well as the attacking path of a bishop at (3,1).}
    \label{fig:ringimage}
\end{figure}
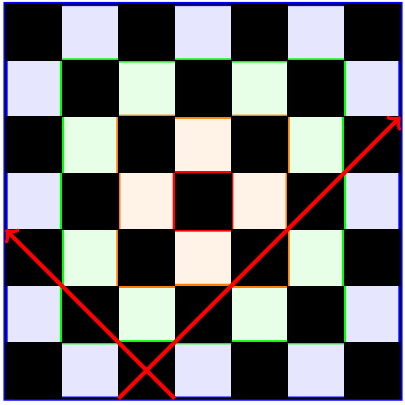

For all $r\neq 0$, the $r$\textsuperscript{th} ring has width $2r+1$ and therefore perimeter $4(2r+1)-4=4(2r)$, given that spaces overlap at the corners. Note that the space in the 0\textsuperscript{th} ring attacks $2(n-1)$ spaces and occupies one. Additionally, we see that for each ring movement outward, a bishop attacks $2$ less spaces in the direction it moves and attacks no more in the opposite direction. Movement within a ring does not change the number of spaces a bishop attacks. Therefore, any bishop in the $r$\textsuperscript{th} ring attacks $2n-2r-1$ spaces. \par
\vspace{.25cm}

Thus, the expected percentage of safe spaces on an odd $n \times n$ chessboard with $n$ bishops is:
\be  \mu_n \ = \ \frac{1}{n^2} \cdot \frac{{\binom{n^2-2n+1}{n}}}{{\binom{n^2}{n}}} + \sum_{r=1}^{(n-1)/2}\frac{4(2r)}{n^2} \frac{{\binom{n^2-2n+2r+1}{n}}}{{\binom{n^2}{n}}} . \ee

Before we evaluate the limit as $n$ approaches infinity, we provide naive bounds. When a bishop is placed in the center ring, with $r=0$, it attacks $2n-1$ spaces. If all bishops were this powerful, then by Lemma \ref{lemma: comboliminf}, the expected percentage of safe spaces would be $1/e^2$. Any bishop placed on the outermost ring, such that $r=(n-1)/2$, attacks $2n-2((n-1)/2)-1=n$ spaces. If all bishops were this powerful, the expected percentage of safe spaces would be $1/e$. Because randomly placed bishops are at most as powerful as a center bishop, and at least as powerful as an outer bishop,  we bound the expected percentage of safe spaces by
\begin{equation}
\frac{1}{e^2} \ \leq \ \mu_n \ \leq \ \frac{1}{e}.
\end{equation}

With these bounds in mind, we now evaluate the limit. We begin by simplifying. We first note that the probability of being placed on the center space is $1/n^2$ and a bishop in that space attacks $2n-1$ spaces, meaning that as $n$ approaches infinity, this bishop does not contribute to the main term.  Additionally, as per Lemma \ref{lemma: comboliminf}, any term of degree less than $k-1$ becomes negligible as $n$ approaches infinity. As we take the limit, we apply Lemma \ref{lemma: comboliminf} so that 
\be \lim_{n\to \infty} \sum_{r=1}^{(n-1)/2} \left(
\frac{8r}{n^2} \frac{{\binom{n^2-2n+2r}{n}}}{{\binom{n^2}{n}}}  \right) \ = \ \lim_{n \to \infty} \sum_{r=1}^{(n-1)/2} \left(
\frac{8r}{n^2} e^{-2(1-r/n)} \right).\ee
Interpreting this as a Riemann sum, with $x=r/n$, we can rewrite this as an integral, and then compute via integration by parts: 

\begin{align*}
\lim_{n \to \infty} \sum_{r=1}^{(n-1)/2} \left(
\frac{8r}{n^2} e^{-2(1-r/n)}\right)
&= \int_{0}^{1/2}8xe^{-2(1-x)}dx\\ 
&= 8e^{-2} \int_{0}^{1/2}xe^{2x}dx\\
&= 8e^{-2}\left( \frac{1}{2}xe^{2x}|_{0}^{1/2} - \frac{1}{2}\int_{0}^{1/2}e^{2x}dx \right)\\
&= 8e^{-2}\left(\frac{e}{4}- \left(\frac{e}{4}-\frac{1}{4}\right)\right)\\
&= \frac{2}{e^2}.
\end{align*}

Because we showed that constant terms do not impact the final limit, the expected proportion of safe spaces is 
\be \lim_{n \to \infty} \sum_{r=1}^{(n-1)/2} \left(
\frac{8r}{n^2} \prod_{\alpha=0}^{2n-2r-1}\frac{n^2-n-\alpha}{n^2-\alpha}  \right) \ = \ \frac{2}{e^2} \ \approx \ 27.067 \% .\ee

We can define the even case similarly. The $0$\textsuperscript{th} ring has $4$ spaces, and the $r$\textsuperscript{th} ring has $4(2r+1)$ spaces. Since we index from $0$, the outermost ring is the $(n/2-1)$\textsuperscript{st} ring. A bishop in the $0$\textsuperscript{th} ring attacks $n-1$ spaces along one diagonal and $n-2$ along the other. For every ring movement outward, it attacks $2$ less spaces. Thus, the expected percentage of safe spaces on an even $n \times n$ chessboard with $n$ randomly placed bishops is
\begin{align}
    \mu_n & \ = \  \sum_{r=0}^{(n/2-1)}\frac{4(2r+1)}{n^2} \frac{{\binom{n^2-2n+2r+2}{n}}}{{\binom{n^2}{n}}}. \notag \\
\end{align}
Similar computation as in the odd case shows that as $n$ approaches infinity, this sum converges to $2/e^2$. \par

\vspace{.25cm}

Note that the spaces in which the rook movement of a queen and the bishop movement of a queen attack are disjoint sets. A queen in the $r$\textsuperscript{th} ring attacks $2(n-1)$ spaces with rook movement and $2(n-1)-2r$ with bishop movement. From this, it follows that the percentage of safe spaces on an $n \times n$ board with $n$ randomly placed queens is \be \sum_{r=0}^{(n-1)/2}\frac{4(2r)}{n^2} \frac{{\binom{n^2-4n+2r+1}{n}}}{{\binom{n^2}{n}}}. \ee
Following the proof above, we observe that the limit of this expression as $n$ approaches infinity is $2/e^4$.

\subsection{Variance for 2D Bishops and Queens} As we determine variance throughout the paper, we use a similar method for all pieces in any $k$ dimensions. To reduce repetition, we prove a general version here, noting that it applies to bishops and queens in two dimensions.
\begin{thm}\label{variance}
    Let $n,k,m,d,a \in \Z^{>0}$. Define $\mu_{n,k}$ as the average percentage of safe spaces on a $k$-dimensional chessboard with side length $n$, with $dn^{k-m}$ attacking pieces placed, each of which attack $an^{m}$ spaces.  Then, the variance of the random variable with mean $\mu_{n,k}$ approaches $0$ as $n$ approaches infinity.
\end{thm}
\begin{proof}
By the definition of variance and standard properties
\begin{align}
    &\Var\left(\frac{S_{n,k}}{n^k}\right) \ = \ \frac{\Var(S_{n,k})}{n^{2k}} \notag \\
    &= \ \frac{1}{n^{2k}}\left( \sum_{i_1,i_2,...,i_k=1}^n\Var(X_{i_1,i_2,...,i_k})+\sum_{\substack{i_1,...,i_k,j_1,...,j_k=1 \\ (i_1,...,i_k) \neq (j_1,...,j_k)}}^n\Cov(X_{i_1,i_2,...,i_k}, X_{j_1,j_2,...,j_k})\right).
\end{align}
We then look at the variance and covariance terms separately.\par
\vspace{.5cm}

\textbf{Variance Term} From the definition of variance, and because $X_{i_1,...,i_k}$ is a binary indicator function, \vspace{-.15cm}\begin{align*}
    \Var(X_{i_1,...,i_k})& \ = \ \E(X_{i_1,...,i_k}^2)-\E(X_{i_1,...,i_k})^2\\
    &= \ \E(X_{i_1,...,i_k})-\E(X_{i_1,...,i_k})^2\\
    &= \ \mu_{n,k}-\mu_{n,k}^2.
\end{align*}
Thus, the contribution of the variance term to $S_{n,k}$ is 
\be \sum_{i_1,i_2,...,i_k=1}^n\Var(X_{i_1,i_2,...,i_k}) \ = \ n^k(\mu_{n,k} - \mu_{n,k}^2).\ee

\textbf{Covariance Term}
We now consider the covariance term. $X_{i_1,...,i_k}$ and $X_{j_1,...,j_k}$ both consider all possible board configurations, so have the same expected value. Then the following holds: \begin{align}
    &\Cov(X_{i_1,...,i_k}, X_{j_1,...,j_k}) \notag\\
    &= \ \E[(X_{i_1,...,i_k}-\mu_{i_1,...,i_k})(X_{j_1,...,j_k}-\mu_{j_1,...,j_k})] \notag \\ \notag
    &= \  \E[X_{i_1,...,i_k}X_{j_1,...,j_k}]-\E[\mu_{j_1,...,j_k}X_{i_1,...,i_k}]-\E[\mu_{i_1,...,i_k}X_{j_1,...,j_k}]+\mu_{i_1,...,i_k}\mu_{j_1,...,j_k}\\
    &= \  \E[X_{i_1,...,i_k}X_{j_1,...,j_k}]- \mu_{n,k}^2.
\end{align}
As there are $n^k$ choices for $(i_1,...,i_k)$ and $n^k-1$ choices for $(j_1,...,j_k)$, there are $n^{2k}-n^k$ distinct pairs. We consider two cases.\par
\vspace{.25cm}

\textit{Case 1: Pieces do not attack each other.}
First, we calculate how many such pairs there are. There are $n^k$ choices for $(i_1,...,i_k)$, which attack some $a_in^m-b_in^{k-c}$ spaces. Then there are $n^k-a_in^m+b_in^{k-c}$ choices for $(j_1,...,j_k)$, meaning there are $n^k(n^{k}-a_in^m+b_in^{k-c})$ pairs that do not attack each other. Given that every piece attacks in some dimension less than $k$, the number of non-attacking pairs is always of degree $2k$, regardless of the dimension or type of piece.

A piece at $(i_1,...,i_k)$ attacks $a_in^m-b_in^{k-c}$ spaces, while a piece at $(j_1,...,j_k)$ attacks $a_jn^m-b_jn^{k-c}$ spaces. Then the probability of a space being safe for a given $(i_1,...,i_k)$ and $(j_1,...,j_k)$ is ${\binom{n^k-a_in^m-a_jn^m+b_in^{k-c}+b_jn^{k-c}}{n^{k-c}}}/{\binom{n^k}{n^{k-c}}}$. This is the number of board configurations where pieces are placed at spaces that attack neither $(i_1,...,i_k)$ nor $(j_1,...,j_k)$, divided by all possible board configurations. Additionally, the probability of a bishop or queen being placed in the $r$\textsuperscript{th} ring is $8r/n^k$. We evaluate the limit
\begin{align}
&\lim_{n \to \infty}\sum_{r_{k},r_{k}=0}^{(n-1)/2,(n-1)/2} \ldots \sum_{r_2,r_2=0}^{r_3,r_3}\left(
\frac{64r_2r_2}{n^{2k}} \prod_{\alpha=0}^{a_in^m+a_jn^m-b_in^{k-c}-b_jn^{k-c}}\frac{n^k-n-\alpha^{k-c}}{n^k-\alpha n^{k-c}}  \right) \notag \\
 &=\lim_{n \to \infty}\left[\sum_{r_{k}=0}^{(n-1)/2}\ldots \sum_{r_2=0}^{r_3}\left(
    \frac{8r_2}{n^{k}} \prod_{\alpha=0}^{a_in^m-b_in^{k-c}}\frac{n^k-n-\alpha}{n^k-\alpha}  \right)\right] \cdot \left[ \sum_{r_{k}=0}^{(n-1)/2}\ldots\sum_{r_2=0}^{r_3}\left(
    \frac{8r_2}{n^{k}} \prod_{\alpha=0}^{a_jn^m-b_jn^{k-c}}\frac{n^k-n-\alpha^{k-c}}{n^k-\alpha n^{k-c}}  \right) \right]\notag \\
&= \ \mu_n \cdot \mu_n \notag \\
&= \  \mu_n^2.
\end{align}

\vspace{.5cm}
\textit{Case 2: Pieces attack each other.}
There are $n^{2k}-n^k$ distinct pairs of spaces. For pieces that attack $an^m$ spaces, we have $n^{2k}-an^{k+m}$ pairs that do not attack each other, meaning that there are $an^{k+m}-n^k$ pairs that do attack each other. However, as we take the variance, we divide by $n^{2k}$, meaning that any terms of degree less than $n^{2k}$ tend to 0, so it is not necessary to calculate the percentage of safe spaces when two pieces attack each other. Therefore, we do not consider it as a term when we calculate the variance.\par
\vspace{.25cm}

\textbf{Variance Conclusion}
We calculate $\Var(S_{n,k}/n^k)=\Var(S_{n,k})/n^{2k}$.  Then,
from the previous sections, \begin{align}
    \lim_{n \to \infty}\frac{\Var(S_{n,k})}{n^{2k}}& \ = \ \lim_{n \to \infty}\frac{1}{n^{2k}}\left(  \sum_{i_1,...,i_k=1}^n\Var(X_{i_1,...,i_k})+\sum_{i_1,...,i_k,j_1,...,j_k=1}^n\Cov(X_{i_1,...,i_k}, X_{j_1,...,j_k})\right) \notag \\ \notag
    &= \  \lim_{n \to \infty} \frac{1}{n^{2k}}\left( n^k(\mu_{n,k} - \mu_{n,k}^2)+(n^{2k}-k_in^{k-c}n^k)\mu_{n,k}^2 - n^k(n^k-1)\mu_{n,k}^2)\right)\\
    &= \ \lim_{n \to \infty}(\mu_{n,k} - \mu_{n,k}^2)\left(\frac{1}{n^k}\right) + \mu_{n,k}^2\left(1-\frac{k_jn^{2k-c}}{n^{2k}}\right) - \mu_{n,k}^2\left(1-\frac{1}{n^k}\right).
\end{align}
We note the constant terms cancel out, given that $\mu_n^2-\mu_n^2=0$. All the other terms approach $0$ as $n$ approaches infinity, so the variance can be made as close to $0$ as desired, meaning the average percentage of safe spaces converges for $n^{k-1}$ line-pieces and $n$ hyper-pieces in any $k$ dimensions.
\end{proof}
It follows that the percentage of safe spaces for bishops and queens also converges in two dimensions.

\section{Hyper-Pieces in Higher Dimensions}

One way to extend the movement of chess pieces to $k$ dimensions is by exploring ``hyper-pieces.'' Any hyper-piece in $k$ dimensions attacks in $(k-1)$-dimensional subspaces. We place $n$ hyper-pieces in order to obtain dominance of the board. We are able to find the expected proportion of safe squares for  hyper-rooks, but the other pieces evade our efforts beyond simple bounding arguments.

\subsection{Hyper-rooks}
We define hyper-rooks and give the proof of Theorem \ref{rooksAvg}.
\begin{defi}[Hyper-rooks]\label{HRdefi}
    A hyper-rook attacks any piece that shares at least one coordinate with it.
\end{defi}
Therefore, a hyper-rook attacks in a $(k-1)$-dimensional plane, which has $n^{k-1}$ spaces. Because the board has $k$ dimensions, there are $k$ of these planes, noting that we have counted some $an^{k-2}$ extra points from overlaps. This means that for $k$ dimensions, a hyper-rook leaves $n^k-kn^{k-1}-an^{k-2}$ safe spaces. Therefore, by Lemma \ref{lemma: comboliminf}, the expected proportion of safe spaces on a $k$-dimensional board with side length $n$ and $n$ hyper-rooks as $n$ approaches infinity is
\be \lim_{n \to \infty}\binom{n^k- kn^{k-1}-an^{k-2}} {n} \bigg/ \binom{n^k} {n} \ = \ \frac{1}{e^k}.\ee

\subsection{Hyper-bishops}

One way to analyze a bishop at $(i,j)$ is as a piece that attacks along the lines of
\begin{align}
    (x-i) + (y-j) \ &= \ 0, \notag \\
    (x-i) - (y-j) \ &= \ 0.
\end{align}
One option to extend this into higher dimensions is to add on extra coordinates, in all possible diagonal subspaces. For example, in $3$ dimensions, the planes to describe hyper-bishop movement at $(i,j,k)$ would be
\begin{align}
    (x-i) + (y-j) + (z-k) \ &= \ 0, \notag \\
    (x-i) + (y-j) - (z-k) \ &= \ 0, \notag \\
    (x-i) - (y-j) + (z-k) \ &= \ 0, \notag \\
    (x-i) - (y-j) - (z-k) \ &= \ 0.
\end{align}
\begin{defi}[Hyper-Bishops]\label{HBdefi}
    In general, for a $k$-dimensional chessboard, a hyper-bishop at\\
    $(a_1, a_2, \ldots, a_k)$ can attack a space \((x_1, x_2, \ldots x_k)\) if it satisfies an equation of the form 
\begin{equation}
    (x_1 - a_1) \pm (x_2 - a_2) \pm \cdots \pm (x_k - a_k) \ = \ 0.
\end{equation}

\end{defi}
\vspace{.25cm}

This definition has a number of key analogies to $2$-dimensional bishops. It only attacks spaces of the same color, and it projects downward, so that in any $2$-dimensional subspace of the board, it moves as a $2$-dimensional bishop would. Unfortunately, the one major disadvantage of this definition is that counting the number of spaces seen in $k$ dimensions becomes challenging as we lose the symmetrical pattern of rings that we saw for bishops in 2 dimensions. Instead, we naively bound the number of spaces attacked for $3$ dimensions, by considering the number of spaces seen by a center piece, which is the most powerful, and a corner piece, which is the least powerful.\par
\vspace{.25cm}

To do this, we consider a vertical slicing method, where we analyze the spaces a hyper-bishop attacks on each $2$-dimensional sub-board. As an example, we show center hyper-bishop movement on a $5 \times 5 \times 5$ cube in Figure \ref{fig:cenbisimage} and corner hyper-bishop movement on a $5 \times 5 \times 5$ cube in Figure \ref{fig:corbisimage}. \par \vspace{.25cm}

\begin{figure}[h!tb]
\centering
\begin{tikzpicture}[scale=.5]
\draw (0,0) rectangle (5,-5);
    \foreach \row in {0,1, ..., 4} {
        \foreach \column in {0, ..., 1} {
    \fill[black] ({2*\column + mod(\row,2)}, -\row) rectangle +(1,-1);
        }
        \foreach \column in {4} {
        \foreach \row in {0,2,4}
    \fill[black] ({\column}, -\row) rectangle +(1,-1);
        }
    }
\draw[red, ultra thick, -] (0,-2)--(3,-5);
\draw[red, ultra thick, -] (5,-2)--(2,-5);
\draw[red, ultra thick, -] (0,-3)--(3,0);
\draw[red, ultra thick, -] (5,-3)--(2,0);

\node at (2.5,-6) {$y = 1$};
\end{tikzpicture}
\begin{tikzpicture}[scale=.5]
\draw (0,0) rectangle (5,-5);
    \foreach \row in {0,1, ..., 4} {
        \foreach \column in {1,3} {
        \foreach \row in {0,2,4}
    \fill[black] ({\column}, -\row) rectangle +(1,-1);
        }
        \foreach \column in {0,2,4} {
        \foreach \row in {1,3}
    \fill[black] ({\column}, -\row) rectangle +(1,-1);
        }
    }
\draw[red, ultra thick, -] (0,-1)--(4,-5);
\draw[red, ultra thick, -] (5,-1)--(1,-5);
\draw[red, ultra thick, -] (0,-4)--(4,0);
\draw[red, ultra thick, -] (5,-4)--(1,0);

\node at (2.5,-6) {$y = 2$};
\end{tikzpicture}
\begin{tikzpicture}[scale=.5]
\draw (0,0) rectangle (5,-5);
    \foreach \row in {0,1, ..., 4} {
        \foreach \column in {0, ..., 1} {
    \fill[black] ({2*\column + mod(\row,2)}, -\row) rectangle +(1,-1);
        }
        \foreach \column in {4} {
        \foreach \row in {0,2,4}
    \fill[black] ({\column}, -\row) rectangle +(1,-1);
        }
    }
\draw[red, ultra thick, -] (0,0)--(5,-5);
\draw[red, ultra thick, -] (5,0)--(0,-5);
\node at (2.5,-6) {$y = 3$};
\end{tikzpicture}
\begin{tikzpicture}[scale=.5]
\draw (0,0) rectangle (5,-5);
     \foreach \row in {0,1, ..., 4} {
        \foreach \column in {1,3} {
        \foreach \row in {0,2,4}
    \fill[black] ({\column}, -\row) rectangle +(1,-1);
        }
        \foreach \column in {0,2,4} {
        \foreach \row in {1,3}
    \fill[black] ({\column}, -\row) rectangle +(1,-1);
        }
    }

\draw[red, ultra thick, -] (0,-1)--(4,-5);
\draw[red, ultra thick, -] (5,-1)--(1,-5);
\draw[red, ultra thick, -] (0,-4)--(4,0);
\draw[red, ultra thick, -] (5,-4)--(1,0);

\node at (2.5,-6) {$y = 4$};
\end{tikzpicture}
\begin{tikzpicture}[scale=.5]
\draw (0,0) rectangle (5,-5);
    \foreach \row in {0,1, ..., 4} {
        \foreach \column in {0, ..., 1} {
    \fill[black] ({2*\column + mod(\row,2)}, -\row) rectangle +(1,-1);
        }
        \foreach \column in {4} {
        \foreach \row in {0,2,4}
    \fill[black] ({\column}, -\row) rectangle +(1,-1);
        }
    }
\draw[red, ultra thick, -] (0,-2)--(3,-5);
\draw[red, ultra thick, -] (5,-2)--(2,-5);
\draw[red, ultra thick, -] (0,-3)--(3,0);
\draw[red, ultra thick, -] (5,-3)--(2,0);

\node at (2.5,-6) {$y = 5$};
\end{tikzpicture}
\caption{Spaces attacked by a hyper-bishop placed at $(3,3,3)$ on a $5 \times 5 \times 5$ board.}
    \label{fig:cenbisimage}
    \vspace{.5cm}
\begin{tikzpicture}[scale=.5]
\draw (0,0) rectangle (5,-5);
    \foreach \row in {0,1, ..., 4} {
        \foreach \column in {0, ..., 1} {
    \fill[black] ({2*\column + mod(\row,2)}, -\row) rectangle +(1,-1);
        }
        \foreach \column in {4} {
        \foreach \row in {0,2,4}
    \fill[black] ({\column}, -\row) rectangle +(1,-1);
        }
    }
\draw[red, ultra thick, -] (0,0)--(5,-5);
\node at (2.5,-6) {$y = 1$};
\end{tikzpicture}
\begin{tikzpicture}[scale=.5]
\draw (0,0) rectangle (5,-5);
     \foreach \row in {0,1, ..., 4} {
        \foreach \column in {1,3} {
        \foreach \row in {0,2,4}
    \fill[black] ({\column}, -\row) rectangle +(1,-1);
        }
        \foreach \column in {0,2,4} {
        \foreach \row in {1,3}
    \fill[black] ({\column}, -\row) rectangle +(1,-1);
        }
    }
\draw[red, ultra thick, -] (1,0)--(5,-4);
\draw[red, ultra thick, -] (0,-1)--(4,-5);
\draw[red, ultra thick, -] (0,-2)--(2,0);
\node at (2.5,-6) {$y = 2$};
\end{tikzpicture}
\begin{tikzpicture}[scale=.5]
\draw (0,0) rectangle (5,-5);
    \foreach \row in {0,1, ..., 4} {
        \foreach \column in {0, ..., 1} {
    \fill[black] ({2*\column + mod(\row,2)}, -\row) rectangle +(1,-1);
        }
        \foreach \column in {4} {
        \foreach \row in {0,2,4}
    \fill[black] ({\column}, -\row) rectangle +(1,-1);
        }
    }
\draw[red, ultra thick, -] (2,0)--(5,-3);
\draw[red, ultra thick, -] (0,-2)--(3,-5);
\draw[red, ultra thick, -] (0,-3)--(3,0);
\node at (2.5,-6) {$y = 3$};
\end{tikzpicture}
\begin{tikzpicture}[scale=.5]
\draw (0,0) rectangle (5,-5);
     \foreach \row in {0,1, ..., 4} {
        \foreach \column in {1,3} {
        \foreach \row in {0,2,4}
    \fill[black] ({\column}, -\row) rectangle +(1,-1);
        }
        \foreach \column in {0,2,4} {
        \foreach \row in {1,3}
    \fill[black] ({\column}, -\row) rectangle +(1,-1);
        }
    }

\draw[red, ultra thick, -] (3,0)--(5,-2);
\draw[red, ultra thick, -] (0,-3)--(2,-5);
\draw[red, ultra thick, -] (0,-4)--(4,0);
\node at (2.5,-6) {$y = 4$};
\end{tikzpicture}
\begin{tikzpicture}[scale=.5]
\draw (0,0) rectangle (5,-5);
    \foreach \row in {0,1, ..., 4} {
        \foreach \column in {0, ..., 1} {
    \fill[black] ({2*\column + mod(\row,2)}, -\row) rectangle +(1,-1);
        }
        \foreach \column in {4} {
        \foreach \row in {0,2,4}
    \fill[black] ({\column}, -\row) rectangle +(1,-1);
        }
    }
\draw[red, ultra thick, -] (4,0)--(5,-1);
\draw[red, ultra thick, -] (0,-4)--(1,-5);
\draw[red, ultra thick, -] (0,-5)--(5,0);
\node at (2.5,-6) {$y = 5$};
\end{tikzpicture}
\caption{Spaces attacked by a hyper-bishop placed at $(1,1,1)$ on a $5 \times 5 \times 5$ board.}
    \label{fig:corbisimage}
\end{figure}
\label{corbisimage}

We first consider the expected percentage of safe spaces if every hyper-bishop is as effective as one placed in the center space. We consider $n$ vertical slices of the chessboard, and analyze the hyper-bishop movement in each one. In the center slice, the hyper-bishop attacks $2n-1$ spaces. For every slice of distance $i$ from $(n-1)/2$, the hyper-bishop attacks in a diagonal on both sides in the ring $i$. This is analogous to moving as if it were $2$ regular bishops on the vertical slice, both placed in ring $i$. This gives us $2(2n-2i-1)-2$ spaces seen. Therefore, the center hyper-bishop sees the following number of spaces on an $n \times n \times n$ board:
\begin{align}
    2n-1 + 2\sum^{\frac{n-1}{2}}_{i=1} 2(2n - 2i -2) \ &= \ 2n-1 + 8\left(n \cdot \frac{n-1}{2} - \frac{n-1}{2} - \sum^{\frac{n-1}{2}}_{i=1} i\right) \notag \\ \notag
    &= \ 2n - 1 + 4n^2 -4n - 4n + 4 - 8\frac{\frac{n-1}{2}\cdot \frac{n+1}{2}}{2} \\
    &= \ 3n^2 - 6n + 4.
\end{align}
If we assume that all pieces see this maximal number of spaces, by Lemma \ref{lemma: comboliminf}, this gives us an expected percentage of safe spaces of
\be \lim_{n \to \infty}\frac{\binom{n^3- 3n^{2}+6n-4}  {n}}{\binom{n^3} {n}} \ = \ \frac{1}{e^3}.\ee \par
\vspace{.25cm}

We obtain a lower bound by considering the weakest possible placements, with all bishops placed in corners. A hyper-bishop placed in one of the lower corners sees $n$ spaces on the lowest slice. For the $i$\textsuperscript{th} slice above, two of the attacking diagonals shift $i$ spaces outward, while one of the attacking diagonals shifts $i$ spaces inward, meaning that it sees $2(n-i+1) + (i-2)$ spaces, as demonstrated in Figure \ref{fig:corbisimage}. Just as above, there are $n$ vertical slices. Hence, the total number of spaces seen is

\begin{equation}
    n+ \sum_{i=1}^{n-1}(2(n-i)+(i-1)) \ = \ \frac{3}{2}(n^2-n)+1.
\end{equation}

\noindent Therefore, by Lemma \ref{lemma: comboliminf}, our upper bound is now
\be \lim_{n \to \infty}\frac{\binom{n^3- 3(n^2 - n)/2 - 1} {n}}{\binom{n^3}{n}} \ = \ \frac{1}{e^{3/2}}.\ee \par
\vspace{.25cm}

Note that both bounds are dependent on the definition of rings from the second dimension. Because the lack of symmetry has prevented us from defining rings for hyper-bishops in $k$ dimensions, finding general bounds for all dimensions proves challenging. We do note a striking similarity between these bounds and our naive bounds in the $2$-dimensional case: both follow the pattern of $1/e^k$ and $1/e^{k/2}$ for the bounds. We are optimistic about future work using this definition.\par
\vspace{.25cm}

\subsection{Hyper-queens}

Our definition for hyper-queens is once again that of a bishop and a rook placed in the same space, though in this case using the hyper-piece version of both. Unfortunately, our current understanding of hyper-bishops is not sufficient to obtain strong results.
We apply our current bounds on hyper-bishops, ignoring the overlap of the rook and bishop given that their seen spaces are mostly disjoint. This gives us an upper bound on spaces seen of $6n^2 - 9n$ and a lower bound of $\frac{9}{2}(n^2 - n)$. This gives bounds on the proportion of safe spaces in the limit, namely
\begin{equation}
\lim_{n \to \infty}\frac{\binom{n^3- 9(n^2 - n)/2}{n}}{\binom{n^3}{n}} \ = \ \frac{1}{e^{9/2}}
\ \\ \quad\rm{ and } \quad\ \\
\lim_{n \to \infty}\frac{\binom{n^3- 6n^{2}+9n} {n}}{\binom{n^3}{n}} \ = \ \frac{1}{e^6},
\end{equation}
both of which follow from applications of Lemma \ref{lemma: comboliminf}.
 \section{Line-Pieces in Higher Dimensions}
It is also possible to consider pieces that move linearly in higher dimensions. We provide a broad definition of such pieces below.
\begin{defi}[Higher dimensional line-pieces]\label{LPdefi}
For any dimension $k > 1$, a $k$-dimensional line-piece moves in perpendicular planes as it does on a $2$-dimensional chessboard. It can move in any $2$-dimensional plane, with all other coordinates being held constant. When determining the expected number of safe spaces, we place $n^{k-1}$ line-pieces.
\end{defi}
Here we place $n^{k-1}$ line-pieces because the number of pieces required to dominate the board is of order $n^{k-1}$. If we were to place some $dn^{k-2}$ line-pieces, then each attacks some $a \cdot n$ spaces where $a$ is a constant factor, and when $n$ tends to infinity, $dan^{k-1}$ never covers a positive proportion of all $n^k$ spaces.\par
\vspace{.25cm}

\subsection{Line-Rooks}
Here, we both refine our definition of line-rooks, and aim to show Theorem \ref{rooksAvg}, though the proof is almost entirely handled by Lemma \ref{lemma: comboliminf}. The remainder is simply counting the number of spaces a line-rook may attack.

\begin{defi}[Line-Rooks]\label{lineRooksDefi}
    A line-rook attacks any space that shares $n^{k-1}$ planes with it, which is equivalent to having all but one coordinate be equal.
\end{defi}
As an example, we show the movement of a line-rook placed at $(3,3,3)$ on a $5 \times 5 \times 5$ board in Figure \ref{fig:linerook}. The line-rook moves only along the bolded lines. \par
\vspace{.25cm}

\begin{figure}[h!]
    \centering
\begin{tikzpicture}[scale=.5]

	\draw[very thick] (0,5,0) -- (0,5,5);
	\draw[very thick] (5,0,0) -- (5,0,5);
	\draw[very thick] (5,5,0) -- (5,5,5);
	\draw[very thick] (0,5,0) -- (5,5,0);
	\draw[very thick] (5,0,0) -- (5,5,0);
 \draw[very thick] (0,0,0) -- (5,0,0);
 \draw[very thick] (0,0,0) -- (0,5,0);
 \draw[very thick] (0,0,0) -- (0,0,5);
     \draw[very thick]	(0,0,5) -- (5,0,5) -- (5,5,5) -- (0,5,5) -- cycle;

\foreach \y in {1,3,5}{	
 \foreach \x in {1,3,5}{
         \draw[opacity=.25, fill = black] (\x-1,\y-1,5) -- (\x,\y-1,5) -- (\x,\y,5) -- (\x-1,\y,5) -- cycle;
    }
    }
\foreach \y in {2,4}{	
 \foreach \x in {2,4}{
         \draw[opacity=.25, fill = black] (\x-1,\y-1,5) -- (\x,\y-1,5) -- (\x,\y,5) -- (\x-1,\y,5) -- cycle;
    }
    }

\foreach \y in {1,3,5}{	
 \foreach \x in {1,3,5}{
         \draw[opacity=.25, fill = black] (\x-1,5,\y-1) -- (\x,5,\y-1) -- (\x,5,\y) -- (\x-1,5,\y) -- cycle;
    }
    }
\foreach \y in {2,4}{	
 \foreach \x in {2,4}{
         \draw[opacity=.25, fill = black] (\x-1,5,\y-1) -- (\x,5,\y-1) -- (\x,5,\y) -- (\x-1,5,\y) -- cycle;
    }
    }
		
\foreach \y in {1,3,5}{	
 \foreach \x in {1,3,5}{
         \draw[opacity=.25, fill = black] (5,\x-1,\y-1) -- (5,\x,\y-1) -- (5,\x,\y) -- (5,\x-1,\y) -- cycle;
    }
    }
\foreach \y in {2,4}{	
 \foreach \x in {2,4}{
         \draw[opacity=.25, fill = black] (5,\x-1,\y-1) -- (5,\x,\y-1) -- (5,\x,\y) -- (5,\x-1,\y) -- cycle;
    }
    }

    \draw[->, teal, ultra thick] (2.5,2.5,2.5) -- (5,2.5,2.5) node[anchor=north]{};
    \draw[->, teal, ultra thick] (2.5,2.5,2.5) -- (0,2.5,2.5) node[anchor=south west]{};
     \draw[->, orange, ultra thick] (2.5,2.5,2.5) -- (2.5,0,2.5) node[anchor=north east]{};
     \draw[->, orange, ultra thick] (2.5,2.5,2.5) -- (2.5,5,2.5) node[anchor=north east]{};
     \draw[->, blue, ultra thick] (2.5,2.5,2.5) -- (2.5,2.5,5) node[anchor=north east]{};
     \draw[->, blue, ultra thick] (2.5,2.5,2.5) -- (2.5,2.5,0) node[anchor=north east]{};

 \filldraw[opacity=.1, fill = orange] (2.5,0,0)--(2.5,0,5)--(2.5,5,5) -- (2.5,5,0)--(2.5,0,0);
\filldraw[opacity=.1, fill = teal] (0,0,2.5) -- (0,5,2.5) -- (5,5,2.5)--(5,0,2.5)--(0,0,2.5);
\filldraw[opacity=.1, fill = blue] (0,2.5,0) -- (0,2.5,5) -- (5,2.5,5)--(5,2.5,0)--(0,2.5,0);
   \draw[dashed, orange,  thick] (2.5,0,0) -- (2.5,0,5);
	\draw[dashed, orange,  thick] (2.5,5,0) -- (2.5,5,5);
 \draw[dashed, orange,  thick] (2.5,0,0) -- (2.5,5,0);
	\draw[dashed, orange,  thick] (2.5,0,5) -- (2.5,5,5);
 \draw[dashed, teal,  thick] (0,0,2.5) -- (0,5,2.5);
 \draw[dashed, teal,  thick] (5,0,2.5) -- (5,5,2.5);
  \draw[dashed, teal,  thick] (0,5,2.5) -- (5,5,2.5);
    \draw[dashed, teal,  thick] (0,0,2.5) -- (5,0,2.5);
    \draw[dashed, blue,  thick] (0,2.5,0) -- (0,2.5,5);
 \draw[dashed, blue,  thick] (0,2.5,5) -- (5,2.5,5);
  \draw[dashed, blue,  thick] (5,2.5,5)--(5,2.5,0);
    \draw[dashed, blue,  thick] (5,2.5,0)--(0,2.5,0);

\end{tikzpicture}
\caption{Movement of a line-rook placed at $(3,3,3)$ on a $5 \times 5 \times 5$ chessboard.}
\label{fig:linerook}
\end{figure}

For any $(i_1, i_2,\ldots, i_k)$, there are $kn-c$ possible rook placements that attack it, noting that the constant $c$ term from overlap becomes negligible as $n$ approaches infinity. We place $n^{k-1}$ pieces and then divide by the total possible number of board combinations. We use Lemma \ref{lemma: comboliminf} to evaluate \be \lim_{n \to \infty} {\binom{n^k- kn+c} {n^{k-1}}}
\bigg/ {\binom{n^{k}} {n^{k-1}}} \ = \ \frac{1}{e^k}.\ee
\par
\vspace{.25cm}

Note that because of overlap, it does not take all $n^{k-1}$ line-rooks to dominate a $n^{k}$ board. As shown in \cite{Eng97}, for $k=3$, only $n^2/2$ rooks are needed to dominate an $n \times n \times n$ chessboard, disregarding parity as $n$ tends to infinity. With this in mind, we can consider an alternate limit in $3$ dimensions, where instead of placing $n^{2}$ line-rooks, we place $n^{2}/2$ rooks. By Lemma \ref{lemma: comboliminf}, this becomes
\be \lim_{n \to \infty} {\binom{n^3- 3n+c} {n^{2}/2}} \bigg/ {\binom{n^{3}} {n^{2}/2}} \ = \ \frac{1}{e^{3/2}}. \ee

The number of rooks needed to dominate a general $k$-dimensional board remains one of many fascinatng open problems in the world of combinatorics and chess.\par
\vspace{.25cm}
\subsection{Line-Bishops} For line-bishops, we note that a two dimensional line-bishop at $(i,j)$ attacks any $(i\pm c,  j\pm c)$. To extend this to $k$ dimensions, we use the following definition.
\begin{defi}[Line-Bishops in $k$ dimensions]\label{lineBishDefi}
    In $k$ dimensions, a $k$-dimensional line-bishop attacks as a 2-dimensional bishop inside any 2-dimensional plane it resides in, and does not attack any other spaces.
\end{defi}
In $3$ dimensions, this is equivalent to attacking along $6$ lines, $2$ for each of the $xy$, $xz$, and $xy$ planes the bishop lies within. Note that this definition maintains parity, as every space on a $k$-dimensional chessboard is an alternating color, so then every space on any $2$-dimensional board within the larger board must also alternate color.\par
\vspace{.25cm}

If we consider all planes that the line-bishop attacks on a $k$-dimensional chessboard, we find that we must take all combinations of $2$ dimensions out of the total $k$ dimensions. Hence, there are $(k^2-k)/2$ two-dimensional planes for a line-bishop to attack along. As we have defined it, a line-bishop placed in the center attacks as a regular bishop attacking $2n-1$ spaces in each of these $(k^2-k)/2$ planes, for a total of $(k^2-k)n - (k^2-k)/2$ spaces attacked. Because $(k^2-k)/2$ is dwarfed as $n$ approaches infinity, we consider the central bishop as attacking $(k^2-k)n$ spaces.\par
\vspace{.25cm}

We now extend our previous concept of rings to $k$ dimensions, to help analyze the spaces seen by any line-bishop. We begin with our concept of rings in the the second dimension which we call $r_2$. As established earlier, for every increase of one to $r_2$, a bishop attacks two fewer spaces. We then define $k-2$ dimensions of rings from $r_3$ to $r_k$, with any $r_i$ ring existing in an $i$-dimensional subspace of the chessboard. Just as with $r_2$, within each $i$-dimensional subspace there are $\lfloor n/2 \rfloor$ values for $r_i$, and in $k$ dimensions, a piece on a board has some value for each $r_2,r_3,..., r_k$. The value of any $r_i$ is determined by the piece's distance from the center point of some $i$-dimensional subspace of the board; if there are multiple subspaces with $i$ dimensions, we define $r_i$ to be the innermost ring, so it takes the minimal value over all $i$-dimensional subspaces. This ensures that bishops placed centrally are more powerful than bishops placed in the outer rings.\par
\vspace{.25cm}
As an example, for a three-dimensional chessboard, a line-bishop placed at the center space along one of the faces of an $n \times n \times n$ chessboard, as depicted in Figure \ref{fig:linebish}, would have $r_2=0$ and $r_3=\lfloor n/2 \rfloor$. Although the bishop is placed in the outermost ring for the $xz$ and $yz$ planes, defining $r_2=0$ ensures that the bishop attacks the correct amount of spaces in the $xy$ plane.\par
\begin{figure}[ht]
    \centering
\begin{tikzpicture}[scale=.5]

	\draw[very thick] (0,5,0) -- (0,5,5);
	\draw[very thick] (5,0,0) -- (5,0,5);
	\draw[very thick] (5,5,0) -- (5,5,5);
	\draw[very thick] (0,5,0) -- (5,5,0);
	\draw[very thick] (5,0,0) -- (5,5,0);
 \draw[very thick] (0,0,0) -- (5,0,0);
 \draw[very thick] (0,0,0) -- (0,5,0);
 \draw[very thick] (0,0,0) -- (0,0,5);
     \draw[very thick]	(0,0,5) -- (5,0,5) -- (5,5,5) -- (0,5,5) -- cycle;

\foreach \y in {1,3,5}{	
 \foreach \x in {1,3,5}{
         \draw[opacity=.25, fill = black] (\x-1,\y-1,5) -- (\x,\y-1,5) -- (\x,\y,5) -- (\x-1,\y,5) -- cycle;
    }
    }
\foreach \y in {2,4}{	
 \foreach \x in {2,4}{
         \draw[opacity=.25, fill = black] (\x-1,\y-1,5) -- (\x,\y-1,5) -- (\x,\y,5) -- (\x-1,\y,5) -- cycle;
    }
    }

\foreach \y in {1,3,5}{	
 \foreach \x in {1,3,5}{
         \draw[opacity=.25, fill = black] (\x-1,5,\y-1) -- (\x,5,\y-1) -- (\x,5,\y) -- (\x-1,5,\y) -- cycle;
    }
    }
\foreach \y in {2,4}{	
 \foreach \x in {2,4}{
         \draw[opacity=.25, fill = black] (\x-1,5,\y-1) -- (\x,5,\y-1) -- (\x,5,\y) -- (\x-1,5,\y) -- cycle;
    }
    }
		
\foreach \y in {1,3,5}{	
 \foreach \x in {1,3,5}{
         \draw[opacity=.25, fill = black] (5,\x-1,\y-1) -- (5,\x,\y-1) -- (5,\x,\y) -- (5,\x-1,\y) -- cycle;
    }
    }
\foreach \y in {2,4}{	
 \foreach \x in {2,4}{
         \draw[opacity=.25, fill = black] (5,\x-1,\y-1) -- (5,\x,\y-1) -- (5,\x,\y) -- (5,\x-1,\y) -- cycle;
    }
    }
\filldraw[opacity=.1, fill = orange] (2.5,0,0)--(2.5,0,5)--(2.5,5,5) -- (2.5,5,0)--(2.5,0,0);
\filldraw[opacity=.1, fill = teal] (0,0,2.5) -- (0,5,2.5) -- (5,5,2.5)--(5,0,2.5)--(0,0,2.5);
\filldraw[opacity=.1, fill = blue] (0,5,0) -- (0,5,5) -- (5,5,5)--(5,5,0)--(0,5,0);
   \draw[dashed, orange,  thick] (2.5,0,0) -- (2.5,0,5);
	\draw[dashed, orange,  thick] (2.5,5,0) -- (2.5,5,5);
 \draw[dashed, orange,  thick] (2.5,0,0) -- (2.5,5,0);
	\draw[dashed, orange,  thick] (2.5,0,5) -- (2.5,5,5);
 \draw[dashed, teal,  thick] (0,0,2.5) -- (0,5,2.5);
 \draw[dashed, teal,  thick] (5,0,2.5) -- (5,5,2.5);
  \draw[dashed, teal,  thick] (0,5,2.5) -- (5,5,2.5);
    \draw[dashed, teal,  thick] (0,0,2.5) -- (5,0,2.5);

    \draw[->, blue, ultra thick] (2.5,5,2.5) -- (0,5,0) node[anchor=north]{};
    \draw[->, blue, ultra thick] (2.5,5,2.5) -- (0,5,5) node[anchor=north]{};
    \draw[->, blue, ultra thick] (2.5,5,2.5) -- (5,5,0) node[anchor=north]{};
    \draw[->, blue, ultra thick] (2.5,5,2.5) -- (5,5,5) node[anchor=north]{};
    \draw[->, teal, ultra thick] (2.5,5,2.5) -- (5,2.5,2.5) node[anchor=north]{};
    \draw[->, teal, ultra thick] (2.5,5,2.5) -- (0,2.5,2.5) node[anchor=north]{};
    \draw[->, orange, ultra thick] (2.5,5,2.5) -- (2.5,2.5,0) node[anchor=north]{};
    \draw[->, orange, ultra thick] (2.5,5,2.5) -- (2.5,2.5,5) node[anchor=north]{};

\end{tikzpicture}
\caption{Movement of a line-bishop placed at $(3,3,5)$ on a $5 \times 5 \times 5$ chessboard. Here, $r_2=0$ and $r_3=2$.}
\label{fig:linebish}
\end{figure}
\vspace{.25cm}
For each movement out in the $r_i$ ring, the line-bishop attacks $2(i-1)$ fewer pieces. A piece moved outwards along some $c$ axis in the $r_i$ ring will move outward in all of the $i-1$ planes that include $c$ as an axis, and as shown in the case for $2$ dimensions, will attack $2$ less spaces in each of these planes for each movement outward. 

Then, for any bishop in rings $(r_2, \ldots, r_k)$, we must subtract the following sum:
\be s \ := \ 2\sum_{i=2}^{k}(i-1)r_i. \ee

For the sake of concise notation, we refer to this sum as $s$ for the remainder of the line-bishop and line-queen proofs.\par
\vspace{.25cm}

As a check, we note that a bishop in the corner has value $n/2$ for all $r$, meaning it attacks, $(k^2-k)n- n/2(2\sum_{i=2}^{k}(i-1))=(k^2-k)n- n(k^2-k)/2=n(k^2-k)/2$ spaces. This is as expected, because a bishop placed in the corner should attack $n$ spaces for all of the $(k^2-k)/2$ planes it resides in. \par
\vspace{.25cm}

For a $k$-dimensional chess board, there are $2^{k-3}(k-1)k$ two-dimensional faces, see for example \cite{Ban96}. As shown in the $2$-dimensional bishop case, there are $8r_2$ spaces for the $r_2$\textsuperscript{th} ring. We define the expected percentage of safe spaces on a chessboard in $k$ dimensions as

\be \mu_k \ = \ \frac{1}{n^k}\sum_{r_k=0}^{n/2}\sum_{r_{k-1}=0}^{r_k}\cdots\sum_{r_2=0}^{r_3}2^{k-3}(k-1)k8r_2\frac{{\binom{n^k- (k^2-k-\frac{s}{n})n}{n^{k-1}}}}{\binom{n^k} {n^{k-1}}}. \ee

\noindent As $n$ tends to infinity, we can apply Lemma \ref{lemma: comboliminf} to see this is \be \mu_k \ = \ \lim_{n \to \infty}\frac{1}{n^k}\sum_{r_k=0}^{n/2}\sum_{r_{k-1}=0}^{r_k}\cdots\sum_{r_2=0}^{r_3}2^{k}(k-1)kr_2\cdot e^{-k^2+k+\frac{s}{n}}. \ee 

\noindent Again, we can interpret this as an iterated Riemann sum, with each $x_k=r_k/n$, allowing us to simplify this expression to be 
\be \mu_k \ = \ 2^k \cdot e^{k^2-k} \cdot k(k-1) \cdot \int_{0 \leq x_1 < \ldots < x_k \leq 1/2} \exp{(-2\sum_{i=1}^k(k-i)x_i)} dx_1 \ldots dx_k\ee 

As an example, we define and computationally determine the expected percentage of safe spaces as $n$ approaches infinity for $k=3$ below:
\begin{align}
    &\lim_{n \to \infty}\frac{1}{n^3}\sum_{r_3=0}^{n/2}\sum_{r_2=0}^{r_3}48r_2e^{-6+\frac{4r_3+2r_2}{n}}
    \ = \ \frac{-1+9e^2-2e^3}{3e^6} \ \approx \ 2.0929 \%.
\end{align}

We notice an interesting phenomenon here. Observe that in three dimensions, a centrally placed bishop attacks $6n-5$ spaces, as opposed to a rook, which only attacks $3n-2$ spaces. A bishop placed in the corner also attacks $3n-2$ spaces, meaning that in three dimensions, any given line-bishop is at least as powerful as a line-rook. We verify this by evaluating the line-rooks limit in three dimensions, noticing that they leave $\approx 4.9787\%$ of the board safe, compared to the $\approx 2.0929\%$ safe for line-bishops. Moreover, for all $k>2$, line-bishops become stronger than line-rooks. Given that line-bishops change in two coordinates as they move, while line-rooks change only in one, the number of diagonals for a bishops to attack along increases as a quadratic, while the number of straight paths for a rook to attack along only increases linearly.

\subsection{Line-Queens}
Again, the rook movement and the bishop movement of a queen are completely disjoint. Therefore, we keep the definition of rings from the line-bishop definition, and add the $kn$ pieces attacked by the rook movement, resulting in a limit of
\begin{align}
    L &\ := \ \lim_{n \to \infty}\frac{1}{n^k}\sum_{r_k=0}^{n/2}\sum_{r_{k-1}=0}^{r_k}...\sum_{r_2=0}^{r_3}2^{k}(k-1)kr_2\frac{\binom{n^k- (k + k^2-k-\frac{s}{n})n} {n^{k-1}}}{\binom{n^k} {n^{k-1}}} \\
\end{align}

When we apply Lemma 2.1, we find that the additional $kn$ attacked spaces contribute a factor of $1/e^k$. Then, for line-queens, the expected percentage of safe spaces is simply the expected proportion of safe spaces for a line-bishop divided by $e^{k}$. We show $k=3$ as an example below:

\begin{align}
    &\lim_{n \to \infty}\frac{1}{n^3}\sum_{r_3=1}^{n/2}\sum_{r_2=1}^{r_3}48r_2e^{-9+\frac{4r_3+2r_2}{n}} \ = \ \frac{-1+9e^2-2e^3}{3e^9} \ \approx \ 0.1042 \%.
\end{align}

\subsection{Variance of Line-Pieces} We proved earlier that variance always converges to $0$ as $n$ approaches infinity. Then the expected proportions of safe spaces converge to the values stated for all pieces above.
\section{Future work}
It is interesting that our naive bounds for hyper-bishops in $3$ dimensions, which bound the expected proportion of safe spaces between $1/e^3$ and $1/e^{3/2}$, are reminiscent of our naive bounding for $2$ dimensional bishops, which were between $1/e^2$ and $1/e$. While precisely counting the number of spaces attacked is not immediately feasible, providing stronger bounds or bounds for $k$ dimensions would be an important step in further understanding the movement of these pieces. \par
\vspace{.25cm}

This project also lends itself to multiple generalizations. We note here that the appearance of an $e^{-k}$ term is reminiscent of a Poisson distribution. While our work here focuses on the probability of a space being attacked by $0$ pieces, it would be worthwhile to consider whether a Poisson distribution could be used to determine the probability of a space being attacked by $x$ pieces. Another possible generalization to consider is the hybrid case, in which some combination of $r$ rooks, $q$ queens and $b$ bishops are placed on a $k$-dimensional chessboard with side length $n$ so that $r+q+b=n^{k-1}$. We note lastly that other possible definitions of hyper-pieces in $k$ dimensions may be possible, and may lend themselves to an easier analysis. \par 


\vspace{.25cm}


\begin{thebibliography}{99}

\bibitem{Wat04}
Watkins~J. Across the Board: The Mathematics of Chessboard Problems.
  Princeton University Press; 2004.

\bibitem{Pet11}
Petkovic~M. Mathematics and Chesss.
  Dover Recreational Math; 2004.

\bibitem{Bez48}
Bezzel~M. Proposal of 8-queens problem.
  Berliner Schachzeitung. 1848; 3 (363):1848. 

\bibitem{BS09}
Bell~J, Stevens~B. A survey of known results and research areas for n-queens.
  Discrete Mathematics. 2009; 309 (1):1-31. doi: https://doi.org/10.1016/j.disc.2007.12.043.

\bibitem{LV11}
Lemaire~B, Vitushinkiy~P. Placing $n$ non dominating queens on the $n \times n$ chessboard, Part I. 
    French Federation of Mathematical Games. 2011.

\bibitem{MST21}
Miller~S.J., Sheng~H., Turek~D. When Rooks Miss: Probability through Chess.
    The College Mathematics Journal. 2021; 52 (2): 82-83. doi: 10.1080/07468342.2021.1886774. 

\bibitem{CMSS}
Cashman~C., Miller~S.J., Shuffelton~J., Son~D. Hyper-bishops, Hyper-rooks, and Hyper-queens: Percentage of Safe Squares on Higher Dimensional Chess Boards. doi: https://doi.org/10.48550/arXiv.2409.04423

\bibitem{Eng97}
Engel~A. Problem-Solving Strategies.
    New York: Springer; 1997. p. 44-45.

\bibitem{Ban96}
Banchoff~T.F. Beyond the Third Dimension.
    W H Freeman \& Co; 1996. 

\end{thebibliography}
\end{document}